\documentclass[11pt, reqno]{amsart}
\usepackage{enumerate}
\usepackage{amssymb,commath}
\usepackage{stix}
\usepackage{amsthm}
\usepackage{thmtools}
\usepackage{
	nameref,%\nameref
}
\usepackage[colorlinks,linkcolor=black,anchorcolor=blue,citecolor=green]{hyperref}
\declaretheorem[]{theorem}
\declaretheorem[]{lemma}
\declaretheorem[]{corollary}
\declaretheorem[]{proposition}
\declaretheorem[]{claim}
\declaretheorem[numbered=no, name=Theorem A]{ThmA}
\declaretheorem[style=remark]{remark}

\allowdisplaybreaks
\newcommand{\injrad}{\operatorname{inj.rad}}
\newcommand{\conjrad}{\operatorname{conj.rad}}
\newcommand{\Ric}{\operatorname{Ric}}

\newcommand{\sndf}{{I\!I}}

\begin{document}

% \title[short text for running head]{full title}
\title[Parametrized Compactness under Bounded Ricci Curvature] {A Parametrized Compactness Theorem under Bounded Ricci Curvature}

%    Only \author and \address are required; other information is
%    optional.  Remove any unused author tags.

%    author one information
% \author[short version for running head]{name for top of paper}
%    author two information
\author{Xiang Li}
\address{School of Mathematical Sciences, Capital Normal Universiy, Beijing China}
\curraddr{}
\email{lixiang@cciee.cn}
%\thanks{}

%    author two information
\author{Shicheng Xu}
\address{School of Mathematical Sciences, Capital Normal Universiy, Beijing China}
\curraddr{}
\email{shichengxu@gmail.com}
\thanks{This work was supported in part by the National
	Natural Science Foundation of China (Grant No.~11401398) and by Youth Innovative Research Team of Capital Normal University.
}

%    \subjclass is required.

\date{\today}

\begin{abstract} We prove a parametrized compactness theorem on manifolds of bounded Ricci curvature, upper bounded diameter and lower bounded injectivity radius.
\end{abstract}

\maketitle

\section{Introduction}

The goal of this paper is to study the parametrized compactness of non-collapsed manifolds of bounded Ricci curvature. 
Let us first recall the Cheeger-Gromov's convergence and compactness theorem (\cite{Cheegerphdthesis,Cheeger1970Finiteness,GLP1981}, cf. \cite{GW1988,Peters1987,Kasue1989}), which says that the set $\mathcal{M}^n_{Rm}(\lambda,v,D)$ of Riemannian $n$-manfolds whose sectional curvature is bounded by $\lambda$, volume is bounded from below by $v>0$ and diameter is bounded from above by $D$,
is precompact in the $C^{1,\alpha}$-topology. Later the $C^{1,\alpha}$-compactness was generalized by Anderson \cite{Anderson1990} to manifolds of two-sided bounded Ricci curvature and lower bounded injectivity radius. Let $\mathcal{M}^n_{Ric}(\lambda,\rho,D)$ be the set, endowed with the Gromov-Hausdorff distance $d_{GH}$, that consists of Riemannian $n$-manifolds whose Ricci curvature is bounded by $\lambda$, injectivity radius is bounded below by $\rho>0$, and diameter is bounded above by $D$.
\begin{theorem}[Anderson \cite{Anderson1990}]\label{thm-anderson-compactness}
	$\mathcal{M}^n_{Ric}(\lambda,\rho,D)$ is precompact in the $C^{1,\alpha}$-topology, in the sense that any sequence $(M_i,g_i)\in \mathcal{M}^n_{Ric}(\lambda,\rho,D)$ admits a subsequence $(M_{i_1},g_{i_1})$ whose $d_{GH}$-limit is isometric to a $C^{1,\alpha}$-Riemannian manifold $(M,g)$, and there are diffeomorphisms $f_{i_1}:M\to M_{i_1}$ for all sufficient large $i_1$ such that the pullback metric $f_{i_1}^*g_{i_1}$ converges to $g$ in the $C^{1,\alpha}$ topology.
\end{theorem}
Our main result is a parameterized version of Theorem \ref{thm-anderson-compactness}. Let $f:(M,g)\to (N,h)$ be a submersion between Riemannian manifolds. The $f$-vertical (resp. horizontal) distribution $\mathcal V_f$ (resp. $\mathcal H_f$) consists of vectors tangent (resp. perpendicular) to fibers. The second fundamental form $\sndf_f$ of $f$-fibers is defined on the $f$-vertical distribution $\mathcal V_f$ by $$\sndf_f:\mathcal V_f(x)\times \mathcal V_f(x)\to \mathcal H_f(x), \quad \sndf_f(T,T)=(\nabla_TT)^\perp|_x.$$ The integrability tensor $A_f$ associated to a submersion $f$ is defined on the $f$-horizontal distribution $\mathcal H_f$ by  $$A_f:\mathcal H_f(x)\times \mathcal H_f(x)\to \mathcal V_f(x), \quad A_f(X,Y)=[X,Y]^\top|_x\in \mathcal V_f(x).$$
A submersion $f:(M,g)\to (N,h)$ is called a $\delta$-Riemannian submersion if there is $\delta\ge 0$ such that for any horizontal vector $v\in \mathcal H_f$,
\begin{align}\label{condition-almost-submersion}
e^{-\delta}|v|\le |df(v)|\le e^\delta |v|.
\end{align}

Let $\mathcal{F}^{m,n}_{Ric}(\delta,\lambda,\mu,\rho,D)$ be the set that consists of any $\delta$-Riemannian submersion $f:(M,g)\to (N,h)$ such that $(M,g)\in \mathcal{M}^m_{Ric}(\lambda,\rho,D)$, $(N,h)\in \mathcal{M}^n_{Ric}(\lambda,\rho,D)$, and the $C^0$-norm of $\sndf_f$ and $A_f$ is bounded by $\mu$, i.e.,
\begin{align}\label{condition-stability-regularity}
\max\{|\sndf_f|, |A_f|\}\le \mu.
\end{align} 
A (not necessarily continuous) map $F:(X,d_X)\to (Y,d_Y)$ is called an \emph{$\epsilon$-Gromov-Hausdorff approximation} (briefly $\epsilon$-GHA), if 
$$|d_Y(F(x),F(x'))-d_X(x,x')|<\epsilon \text{ for any $x,x'\in X$, and $f(X)$ is $\epsilon$-dense in $Y$}.$$
For any two maps in $\mathcal{F}^{m,n}_{Ric}(\delta,\lambda,\mu,\rho, D)$,  $f_i:(M_i,g_i)\to (N_i,h_i)$ ($i=1,2$),
we say that they are \emph{$\epsilon$-close} in the Gromov-Hausdorff topology if there is a pair of $\epsilon$-GHAs,
$\phi:(M_1,g_1)\to (M_2,g_2)$ and $\psi:(N_1,h_1)\to (N_2,h_2)$, such that 
$$d(\psi\circ f_1,f_2\circ\phi)=
\sup_{x\in M_1}d_{h_2}(\psi\circ f_1(x),f_2\circ \phi(x))
\le \epsilon.$$

\begin{ThmA} 
	Given $m,\lambda,\mu,\rho>0$, there are constants $\delta_0(m,\lambda,\mu,\rho)>0$ and $\epsilon_0(m,\lambda,\mu,\rho)>0$, such that for any $0\le \delta< \delta_0$ and $0<\epsilon<\epsilon_0$, if two elements $f_i:(M_i,g_i)\to (N_i,h_i)\in \mathcal{F}^{m,n}_{Ric}(\delta,\lambda,\mu,\rho,D)$ $(i=1,2)$ are $\epsilon$-close in the Gromov-Hausdorff topology, then there exits a pair of diffeomorphisms $(\Phi,\Psi)$ such that $\Phi:(M_1,g_1)\to (M_2,g_2)$ is $e^{\varkappa(\delta, \epsilon)}$-bi-Lipschitz, $\Psi:(N_1,h_1)\to (N_2,h_2)$ is $e^{\varkappa(\epsilon)}$-bi-Lipschitz, and $\Psi\circ f_1=f_2\circ \Phi$.
\end{ThmA}
In Theorem A  $\varkappa(\delta,\epsilon) = \varkappa(\delta,\epsilon \,|\, m,\lambda,\mu,\rho)$, where (and hereafter) $\varkappa(a,b\,|\,c,\dots)$ is to denote a positive function in $a,b,c,\dots$ such that after fixing the other parameters, $\varkappa(a,b)=\varkappa(a,b\,|\,c,\dots)\to 0$ as $a,b\to 0$.

\begin{remark}\label{rem-no-C0-close}
	In general $\Phi^*g_2$ may be not $\varkappa(\epsilon)$-close to $g_1$ as $\epsilon\to 0$ for $\delta$-Riemannian submersions with $\delta\neq 0$. For example, a $\delta$-Riemannian submersion
	$f:S^1\times S^1\to S^1$, $f(\theta_1,\theta_2)=\theta_2+\tau(\theta_2)$ could be arbitrary close to the canonical projection $(\theta_1,\theta_2)\mapsto \theta_2$ such that (\ref{condition-stability-regularity}) holds with $\mu=0$, while the $C^0$-norm $|\partial_{2}f|_{C^0}\ge 1+\delta_0/2$.
	Then $\Phi$ is given by $(\vartheta_1,\vartheta_2)=\Phi(\theta_1,\theta_2) =(\theta_1,f(\theta_1,\theta_2))$, and $\Phi^*(d\vartheta_1^2+d\vartheta_2^2)$ is definite away from $d\theta_1^2+d\theta_2^2$. Therefore, the conclusion of Theorem A for $\delta$-Riemannian submersions with $\delta>0$ is sharp.
\end{remark}
\begin{remark}\label{rem-C0alpha-close}
	If instead of (\ref{condition-stability-regularity}), $f_i$ satisfies a stronger assumption that the second fundamental form of $f_i$ admits a uniform bound, i.e.,
	\begin{equation}\label{condition-2nd-order}
	|\nabla^2f_i|\le \mu,
	\end{equation}
	then it can be seen that the conclusion of Theorem A holds for any $\delta\ge 0$ and sufficient small $\epsilon\le \epsilon_1(m,\delta,\mu,\rho)$ such that $\Phi^*g_2$ and $g_1$ are $\varkappa(\epsilon)$-close in the $C^{0,\alpha}$-norm.
\end{remark}

\begin{remark}
	A partial motivation to consider $\delta$-Riemannian submersions is that they naturally arise as a manifold is collapsed under bounded sectional curvature (cf. \cite{Fukaya1987Collapsing,Fukaya1988,CFG1992}). We actually prove the diffeomorphic stability in Theorem A for $\delta$-Riemannian submersions whose total spaces are allowed to be collapsed under bounded Ricci curvature, such that the conjugate radius has a positive lower bound, and the bound in (\ref{condition-stability-regularity}) blows up at the rate proportional to $\epsilon^{-1}$; see Theorem \ref{thm-tech}.  A stability result (Proposition A.2.2 in \cite{CFG1992}) for $\delta$-Riemannian submersions with higher regularities $|\nabla^jf|\le \mu_j$ was proved and applied in Cheeger-Fukaya-Gromov's construction \cite{CFG1992} of the nilpotent Killing structure on manifolds of bounded sectional curvature. Similar techniques in this paper can be applied to study the stability of $N$-structures on a collapsed manifold; see \cite{JiangLiXu17}.
\end{remark}

By Theorem A, $\mathcal{F}^{m,n}_{Ric}(\delta,\lambda,\mu,\rho,D)$ with $\delta<\delta_0$ contains only finitely many diffeomorphic isomorphism classes of fiber bundles.
In general the limit of a sequence of $\delta$-Riemannian submersions, however, is only an $e^\delta$-Lipschitz-co-Lipschitz map (briefly, LcL) that may not be smooth.  Recall that (cf. \cite{Rong2012Stability})
for any $Q\ge 1$, a map $f:(X,d_X)\to (Y,d_Y)$ between metric spaces is called $Q$-LcL if 
$$B_{Q^{-1}r}(f(x))\subset f(B_r(x)) \subset B_{Qr}(f(x)),\quad \forall\, x\in X,\, r>0.$$
A $1$-LcL is called a submetry (cf. \cite{Berestovskii2000}).
Clearly, a proper submersion $f:(M,g)\to (N,h)$ is a $\delta$-Riemannian submersion if and only if it is $e^\delta$-LcL. 

By Theorem A, if $f_1,f_2\in \mathcal{F}^{m,n}_{Ric}(\lambda,\mu,\rho,D)=\mathcal{F}^{m,n}_{Ric}(0,\lambda,\mu,\rho,D)$ are Riemannian submersions that are $\epsilon$-close to each other, then the pullback metric $\Phi^*g_2$ is $\varkappa(\epsilon)$-close to $g_1$ in the $C^{0}$-norm and $\Psi^*h_2$ is $\varkappa(\epsilon)$-close to $h_1$ in the $C^{1,\alpha}$-norm. Base on this, we are able to prove that the limit map of Riemannian submersions in $\mathcal{F}^{m,n}_{Ric}(\lambda,\mu,\rho,D)$ is also smooth. Therefore, as a parametrized version of Theorem \ref{thm-anderson-compactness}, the $C^0$-convergence and compactness holds for Riemannian submersions in $\mathcal{F}^{m,n}_{Ric}(\lambda,\mu,\rho,D)$.

\begin{corollary}\label{cor-C0-compactness}
	The set $\mathcal{F}^{m,n}_{Ric}(\lambda,\mu,\rho,D)$ is precompact in the sense that any sequence of Riemannian submersions, $f_i:(M_i,g_i)\to (N_i,h_i)$, admits a subsequence $f_{i_1}$ that converges to a 
	smooth Riemannian submersion $f:(M,g)\to (N,h)$ between $C^{1,\alpha}$-Riemannian manifolds, and there are 
	diffeomorphisms $\Phi_{i_1}:M\to M_{i_1}$, $\Psi_{i_1}:N\to N_{i_1}$ for all sufficient large $i_1$ such that $\Psi_{i_1} \circ f=f_{i_1}\circ \Phi_{i_1}$, and the pullback metric $\Phi_{i_1}^*g_{i_1}$ converges to $g$ in the $C^{0}$-norm.
\end{corollary}

\begin{remark}
	It is interesting to ask whether the diffeomorphisms in Corollary \ref{cor-C0-compactness} can be chosen for Riemannian submersions such that the convergence of pullback metrics is in the $C^{0,\alpha}$-norm, which is the best that one can expect in general (cf. the expression of $\Phi$ in Remark \ref{rem-no-C0-close}).
\end{remark}

The regularity condition (\ref{condition-stability-regularity}) is redundant for Riemannian submersions between manifolds of bounded sectional curvature in $\mathcal{M}^n_{Rm}(\lambda,v,D)$. Indeed, because any Riemannian submersion $f:(M,g)\to (N,h)$ can be expressed via distance coordinates on $(N,h)$ such that each component is a distance function to a $f$-fiber (see \cite{Pe16}, \cite{Kapovitch2007Perelman}, \cite{Rong2011Stability}), it follows easily from the Hessian comparison and $C^{1,\alpha}$-compactness that $|\nabla^2f|$ depends only on the lower bound of sectional curvature of $(M,g)$ and the injectivity radius of $(N,h)$ (cf. \cite{Berestovskii2000}). Let $\mathcal{F}^{m,n}_{Rm}(\lambda,v,D)$ be the set consisting of all Riemannian submersions between manifolds in $\mathcal{M}^m_{Rm}(\lambda,v,D)$ and those in $\mathcal{M}^n_{Rm}(\lambda,v,D)$. By Remark \ref{rem-C0alpha-close}, we conclude the following $C^{0,\alpha}$-convergence and compactness of $\mathcal{F}^{m,n}_{Rm}(\lambda,v,D)$.
\begin{corollary}\label{cor-riem-submersion-bounded-sec}
	The set $\mathcal{F}^{m,n}_{Rm}(\lambda,v,D)$ is precompact in the $C^{0,\alpha}$-topology, in the sense of Corollary \ref{cor-C0-compactness} such that the pullback metric converges in the $C^{0,\alpha}$-norm.
\end{corollary}

\begin{remark}
	Corollary \ref{cor-riem-submersion-bounded-sec} improves a diffeomorphic finiteness of those Riemannian submersions in $\mathcal{F}^{m,n}_{Rm}(\lambda,v,D)$ proved by Tapp \cite{Ta2000} (cf. \cite{Ta2002}), where the base spaces were required to be simply-connected.
	Tapp's proof was based on a diffeomorphic finiteness of Riemannian submersions with bounded $T$-tensor and $A$-tensor (\cite{ONeill1966}), which is equivalent to (\ref{condition-stability-regularity}), proved in \cite{Wa1992} (as corrected in \cite{Wa1993}) under a different setting where the base spaces and a fiber are of finite diffeomorphism types and bounded sectional curvature.
\end{remark}

In contrast to the diffeomorphic stability, the \emph{homeomorphic} stability/finiteness in Corollary \ref{cor-riem-submersion-bounded-sec} is well known and extensively studied in more general settings. 
As observed by Kapovitch \cite{Kapovitch2007Perelman}, Perelman's parameterized stability theorem \cite{Per1991}, which plays an essential role in the proof of his stability theorem \cite{Per1991,Kapovitch2007Perelman}, directly implies the homeomorphic compactness of $1$-LcL maps (i.e., submetries) from non-collapsed Alexandrov spaces with curvature bounded below to non-collapsed Riemannian manifolds of bounded sectional curvature. In particular, it covers the homeomorphic compactness of Riemannian submersions proved by Wu \cite{Wu1996} and Tapp \cite{Ta2002}. Similar parameterized stability under lower bounded (sectional) curvature was extended to $\delta$-Riemannian submersions in \cite{Rong2012Stability}.

\begin{theorem}[\cite{Rong2012Stability}]\label{thm-homeo-stability}
	There is $\delta_0(m,v,\rho,D)>0$ such that for any two $\delta_0$-Riemannian submersions $f_i:(M_i,g_i)\to (N_i,h_i)$ ($i=1,2$) whose dimension, sectional curvature, volume and injectivity radius
	\begin{align*}
	\dim M,\dim N\le m,\qquad &\sec(M_i,g_i), \sec(N_i,h_i)\ge \lambda>0,\\ 
	\operatorname{vol}(M_i,g_i)\ge v>0, \qquad &\injrad(N_i,h_i)\ge \rho>0,
	\end{align*} if $f_1$ and $f_2$ are $\epsilon$-close in $d_{GH}$, then there is a pair of homeomorphisms $(\Phi,\Psi)$ such that $\Psi$ and $\Psi$ are $\varkappa(\delta_0,\epsilon|m,v,\rho,D)$-GHAs satisfying $\Psi\circ f_1=f_2\circ \Phi$.
\end{theorem}
We refer to \cite{Rong2012Stability,Rong2011Stability} for more stability results about $e^\delta$-LcL maps and $\delta$-submetries (weaker than LcL, see \cite{Rong2011Stability}) with certain regularities.

Compared with Theorem \ref{thm-homeo-stability}, the regularity condition (\ref{condition-stability-regularity}) on $\sndf_f$ and $A_f$ in Theorem A is naturally required in order to derive the diffeomorphisms.  We do not know, however, whether the homeomorphic stability in Theorem A holds without assuming that $|\sndf_f|$ and $|A_f|$ are bounded. 

The conclusion of Theorem A would fail if the regularity conditions (\ref{condition-almost-submersion}) and (\ref{condition-stability-regularity}) are removed. For example, let us consider the collection $\mathcal P$ of isomorphism classes of circle bundles from $M=S^2\times S^3$ to $N=S^2\times S^2$, then it contains infinitely many pairwisely non-isomorphic circle bundles, which are represented by elements $(p,q)$ in the  cohomology  $H^2(S^2\times S^2, \mathbb Z)\cong \mathbb Z\oplus\mathbb Z$, with $p,q$ co-prime integers. By Theorem A,  for the metrics on $S^2\times S^3$ and $S^2\times S^2$ lying in $\mathcal{M}^n_{Ric}(\lambda,\rho,D)$, the maps in $\mathcal P$ cannot be (uniformly) $\delta$-Riemannian submersions such that (\ref{condition-stability-regularity}) holds.

Our proof also yields an equivariant version of Theorem A. Let $G$ be a Lie group acting on $M$ and $N$ respectively. A map $f:M\to N$ is called $G$-equivariant if $$f\circ g(x)=g\circ f(x), \quad\text{for } \forall\, g \in G,\; \forall\, x\in M.$$
\begin{theorem}\label{thm-equivariant}
	Let the assumptions be as in Theorem A. If in addition, there is a Lie group $G$ acting isometrically on $(M_i,g_i)$ and $(N_i,h_i)$ respectively such that each $f_i$ is $G$-equivariant, and the closeness of $f_i$ measured by equivariant Gromov-Hausdorff distance is no more than $0\le\epsilon<\epsilon_0(m,\lambda,\mu,\rho)$, then the diffeomorphisms $\Phi,\Psi$ in the bundle isomorphism can be chosen $G$-equivariant.
\end{theorem}

The difficulty in proving Theorem A is the lack of regularity. As mentioned earlier in Remark \ref{rem-C0alpha-close}, if one assumes that (\ref{condition-2nd-order}) holds, then the conclusion of Theorem A easily follows from earlier known results. Indeed, by Theorem \ref{thm-anderson-compactness}  the almost Riemannian submersions $f_i$ in Theorem A could be easily reduced to those between two fixed Riemannian manifolds. Moreover, by the $C^{1,\alpha}$-convergence and $|\nabla^2f_i|\le \mu$, one may assume without loss of generality that $f_1$ and $f_2$ are $\varkappa(\epsilon)$-$C^1$-close, i.e.,
\begin{equation}\label{ineq-C1-close}
|P\circ \dif f_1-\dif f_2|\le \varkappa(\epsilon\,|\, \mu,\rho, m),
\end{equation}
where $P:T_{f_1(x)}N\to T_{f_2(x)}N$ is the parallel transport on $(N, h)$.
It is well known that if two fibrations $f_i:M\to N$ ($i=1,2$) are sufficiently $C^1$-close, then one of them can be deformed onto the other by isotopies (cf. Proposition A.2.2 in \cite{CFG1992}). Therefore, for any (not necessarily small) $\delta\ge 0$ Theorem A holds for $\delta$-Riemannian submersions in $\mathcal{F}^{m,n}_{Ric}(\delta,\lambda,\mu,\rho,D)$, provided that $|\nabla^2f|$ admits a uniform upper bound.

However, the almost Riemannian submersions $f_1$ and $f_2$ in Theorem A generally are not $\varkappa(\epsilon)$-$C^1$-close; see Remark \ref{rem-no-C0-close}. Hence the argument in the previous paragraph fails. Instead of the $C^1$-closeness, we will prove \emph{certain weaker regularity} between $f_1$ and $f_2$. That is, after identifying $M_1$ and $M_2$ (resp. $N_1$ and $N_2$) by diffeomorphisms provided by Theorem \ref{thm-anderson-compactness}, the $f_i$-vertical distributions $\mathcal V_{f_i}$ are $\varkappa(\delta,\epsilon)$-close to each other; see Proposition \ref{prop-vertical-close}. Then we are able to prove that the bundle map constructed via horizontal lifting curves (see (\ref{def-bundle-map})) is a diffeomorphism without knowing that $|P\circ \dif f_1-\dif f_2|$ is small.

The remaining of the paper is organized as follows. In Section 2 we reduce the proofs of Theorem A to a technical Theorem \ref{thm-tech}, where the total space is allowed to be collapsed and the bound in (\ref{condition-stability-regularity}) blows up at the rate of closeness of $f_i$. In Section 3 we prove Theorem \ref{thm-tech} by assuming Proposition \ref{prop-vertical-close}, which asserts that under the assumptions of Theorem \ref{thm-tech}, if two $\epsilon$-Riemannian submersions are $\epsilon$-close, then their vertical distributions are $\varkappa(\epsilon)$-close. The proof of Proposition \ref{prop-vertical-close} is given in Section 4. In Section 5 we prove a crucial estimate in proving Proposition \ref{prop-vertical-close}, which clarifies how the deviation of two curves tangent at the same start point depends on their geodesic curvature explicitly at a definite time.

\section{Proof of Theorem A}

Let $f_i:(M_i,g_i)\to (N_i,h_i)$ ($i=1,2$) be two $\delta$-Riemannian submersions in Theorem A. By Theorem \ref{thm-anderson-compactness}, we may assume that $M_1$ and $M_2$ (respectively, $N_1$ and $N_2$) are diffeomorphic with each other. Then the existence of diffeomorphisms $\Phi,\Psi$ in Theorem A is reduced to the following technical theorem, whose condition is weaker than Theorem A.

\begin{theorem}\label{thm-tech}
	Given positive real numbers $c_0, r_0$ and a positive integer $m\in \mathbb{N}^+$, there is $\epsilon_0(c_0, r_0, m)>0$ such that for all $0\le \epsilon<\epsilon_0(c_0, r_0, m)$ the following holds.
	
	Let $f_i:(M,g)\to (N,h)$ $(i=1.2)$ be two $\epsilon$-Riemannian submersions between connected closed Riemannian manifolds whose dimension $\le m$, Ricci curvature, conjugate radius and injectivity radius satisfy 
	$$\begin{aligned}
	|\Ric(M,g)|\leq 1,\quad &\conjrad(M,g)\ge r_0;
	\\
	|\Ric(N,h)|\le 1,
	\quad &\injrad(N,h)\geq r_0.\end{aligned}
	$$
	Assume that 
	\begin{enumerate}
		\item[(4.1)] $f_1$ and $f_2$ are $\epsilon$-close, i.e.,
		$d_h(f_1,f_2)\le\epsilon,$ and
		\item[(4.2)] $\sndf_{f_i}$, $A_{f_i}$ satisfy the following rescaling invariant control,
		$$|\sndf_{f_i}|\cdot d_h(f_1,f_2)\le c_0, \quad |A_{f_i}|\cdot d_h(f_1,f_2)\le c_0.$$
	\end{enumerate}
	Then there is a diffeomorphism $\Phi:M\to M$ that is a bundle isomorphism between fiber bundles $(M,N,f_1)$ and $(M,N,f_2)$, i.e., $f_2\circ \Phi=f_1$, such that for any $x\in M$ and any vector $v\in T_xM$, 
	\begin{equation}\label{ineq-almost-isometry}
	e^{-\varkappa(\epsilon,c_0|r_0,m)}\cdot |v|\le |\dif \Phi(v)|\le e^{\varkappa(\epsilon,c_0|r_0,m)}\cdot |v|.
	\end{equation}
\end{theorem}

Let us prove Theorem A and Corollary \ref{cor-C0-compactness} by assuming Theorem 2.1.
\begin{proof}[Proof of Theorem A]$~$
	
	Let $f_i:(M_i,g_i)\to (N_i,h_i)$ be two $\delta$-Riemannian submersion that are $\epsilon$-close in Gromov-Hausdorff  distance. That is, there are $\epsilon$-GHAs $\phi:(M_1,g_1)\to (M_2,g_2)$ and $\psi:(N_1,h_1)\to (N_2,h_2)$ such that
	$d(\psi\circ f_1,f_2\circ\phi)\le \epsilon$. By Theorem \ref{thm-anderson-compactness}, $\phi$ and $\psi$ can be replaced by diffeomorphisms, denoted still by $\phi$ and $\psi$, such that the pullback metrics $\phi^*g_2$ and $\psi^*h_2$ are $\varkappa(\epsilon)$-$C^{1,\alpha}$-close to $g_1$ and $h_1$ respectively. Because after a small $C^{1,\alpha}$-perturbation on the metric, $f_i$ is still an $\delta$-Riemannian submersion that satisfies (\ref{condition-stability-regularity}) (with $\delta$ changing a little), we assume that, without loss of generality, $\hat f_1=\psi\circ f_1$ and $\hat f_2=f_2\circ\phi$ are two $\delta$-Riemannian submersions $\hat f_i:(M,g_1)\to (N,h_2)$ such that
	$d_{h_2}(\hat f_1,\hat f_2)\le \epsilon$.  
	
	Now by Theorem \ref{thm-tech}, for $0\le \delta<\epsilon_0(c_0,\rho,m)$ with $c_0=\mu\epsilon$, there is a diffeomorphism $\Phi_1:M_1\to M_1$ such that $\hat f_2\circ \Phi_1=\hat f_1$, and $\Phi_{1}$ is $e^{\varkappa(\delta,\mu\epsilon)}$-bi-Lipschitz. Let $\Phi=\phi\circ\Phi_1$ and $\Psi=\psi$, then 
	$$f_2\circ \Phi=(f_2\circ \phi) \circ \Phi_1=\hat f_2\circ \Phi_1=\hat f_1=\Psi \circ f_1.$$ 
	
	If $f_i$ are Riemannian submersions, then by (\ref{ineq-almost-isometry}), $\dif \Phi_1$ is an $e^{\varkappa(\epsilon,\mu\epsilon)}$-almost isometry, which implies that $\Phi^*g_2$ is $\varkappa(\epsilon)$-$C^0$-close to $g_1$. 
\end{proof}

\begin{proof}[Proof of Corollary \ref{cor-C0-compactness}] $~$
	
	By Theorem \ref{thm-anderson-compactness} and Arzel\`a-Ascoli theorem, any sequence of Riemannian submersion $f_i:(M_i,g_i)\to (N_i,h_i)$ in $\mathcal{F}^{m,n}_{Ric}(\lambda,\mu,\rho,D)$ admits a subsequence, still denoted by $f_i$, that converges to a submetry $f:(M,g_\infty)\to (N,h_\infty)$ between $C^{1,\alpha}$-Riemannian manifolds. It suffices to translate the convergence $f_i\to f$ to a convergence of metric tensors on $M$ with respect to a fixed smooth fiber bundle projection.
	
	By Theorem \ref{thm-anderson-compactness}, we assume that for each $i$, $(M_i,g_i)=(M,g_i)$ and $(N_i,h_i)=(N,h_i)$ such that $g_i$ (resp. $h_i$) converges to $g_\infty$ (resp. $h_\infty$) in the $C^{1,\alpha}$-norm with respect to $g_0$.
	Let $\Phi_{i,j}$ be the diffeomorphism in Theorem A such that $f_{j}\circ \Phi_{i,j}=f_i$, then the pullback metric $\Phi_{i,j}^*g_{j}$ satisfies
	$$|\Phi_{i,j}^*g_{j}-g_i|_{g_0}\to 0,\quad \text{as $i,j\to \infty$,}$$
	and each $f_{j}$ can be represented by $$f_0:(M,(\Phi_{j-1,j}\cdots\circ \Phi_{1,2}\circ \Phi_{0,1})^*g_j)\to (N,h_{j}).$$ By passing to a subsequence, we may assume that
	$$|\Phi_{i,i+1}^*g_{i+1}-g_i|_{g_0}\le 2^{-i-1}.$$
	Then it is easy to see that $(\Phi_{j-1,j}\cdots\circ \Phi_{1,2}\circ \Phi_{0,1})^*g_j$ converges in $C^0$-norm to $g_\infty$ as $j\to \infty$, such that the limit map $f_\infty$ coincides with the smooth submersion $f_0:(M,g_\infty)\to (N,h_\infty)$.
\end{proof}

In the end of this subsection we give an elementary estimate on variations of horizontal lifting curves of a submersion that will be used in the proof of Theorem \ref{thm-tech}. 
\begin{lemma}\label{lem-variation}
	Let $f:(M,g)\to (N,h)$ be an $\epsilon$-Riemannian submersion.
	\begin{enumerate}
		\item[(1.1)] Let $\gamma:[0,1]\to N$ be a minimal geodesic from $p$ to $q\in N$. Let $\tilde \gamma(s,t)$ be a variation in $M$ such that $\tilde \gamma(s,0)$ lies in $F_p=f^{-1}(p)$, and $\tilde \gamma_s(\cdot)=\tilde \gamma(s,\cdot)$ is the horizontal lifting of $\gamma$. 
		Then the variation field $\tilde\partial_s=\partial_s\tilde \gamma$ satisfies
		$$e^{-e^\epsilon |\sndf_f|\cdot r} |v| \le |\tilde \partial_s|(0,1)\le e^{e^\epsilon |\sndf_f|\cdot r} |v|, $$
		where $v=\tilde \partial_s(0,0)$ and $r=d(p,q)$.
		\item[(1.2)] Assume that $|\sec(N,h)|\le 1$. Let $\alpha:(-\delta,\delta)\times [0,1]\to N$ be a smooth variation of geodesics $\alpha_s(\cdot)=\alpha(s,\cdot)$ of length $\le \frac{\pi}{2}$. Let $\tilde \alpha:(-\delta,\delta)\times [0,1]\to M$ be a lifting variation in $M$ such that $\tilde \alpha(\cdot ,0)$ is a horizontal lifting of $\alpha(\cdot ,0)$ and $\tilde \alpha(s,\cdot )$ the horizontal lifting of $\alpha(s,\cdot)$ starting at $\tilde \alpha(s,0)$. Then the variation field $\tilde \partial_s=\partial_s\tilde \alpha$ satisfies that
		$$|\tilde \partial_s^\top|(0,1)\le C (a+b) e^{3\epsilon+e^\epsilon |\sndf_f|r}|A_f|r,$$
		where $a=|\partial_s\alpha|(0,0)$, $b=|\partial_s\alpha|(0,1)$ and $r=\operatorname{length}(\alpha_0)$.
	\end{enumerate}
	
\end{lemma}

\begin{proof}
	We prove (1.1) first.
	Clearly, the variation vector field along $\tilde\gamma$ satisfies that $\tilde \partial_s\in F_{\gamma(t)}$ is vertical, $\tilde\partial_t=\partial_t\tilde \gamma$ is horizontal, and $|\tilde \partial_t|\le e^\epsilon d(p,q)= e^\epsilon r$.
	By directly calculation,
	\begin{align*}
	\frac{\partial}{\partial t} \frac{1}{2}\left<\tilde \partial_s, \tilde \partial_s\right>
	=-\left<\tilde \partial_t,\nabla_{\tilde \partial_s}\tilde \partial_s\right>
	\le \left<\tilde \partial_s,\tilde \partial_s\right> \cdot |\sndf| \cdot e^\epsilon r 
	\end{align*}
	which implies that
	$$\left<\tilde \partial_s,\tilde \partial_s\right>(t)\le \left<\tilde \partial_s,\tilde \partial_s\right>(0)\cdot e^{2e^\epsilon r|\sndf| t}.$$
	Hence by writing $v=\tilde \partial_s(0,0)$,
	$$|\tilde \partial_s|(0,1)\le e^{e^\epsilon |\sndf_f|\cdot r} |v|.$$
	The other side of the desired inequality follows from symmetry.
	
	Next, we prove (1.2).
	By Rauch comparison for $|\sec|\le 1$, the Jacobian field $\partial_s\alpha$ satisfies
	\begin{equation}\label{ineq-integrable-1}
	|\partial_s\alpha|\le a(1-t)+bt +(a+b)o(r)	\end{equation}
	Because $\tilde \partial_t \perp \tilde \partial_s^\top$ and $[\tilde \partial_t,\tilde \partial_s^\top+\tilde \partial_s^\perp]=0$, we derive
	\begin{equation}\label{ineq-integrable-2}
	\begin{aligned}
	\left<\nabla_{\tilde \partial_t}\tilde \partial_s^\top, \tilde \partial_s^\top\right> &= \left< \nabla_{\tilde \partial_s^\top} \tilde \partial_t, \tilde \partial_s^\top \right>+\left< 
	\left[ \tilde \partial_t, \tilde \partial_s^\top \right],\tilde \partial_s^\top \right>\\
	&=-\left<\tilde \partial_t, \nabla_{\tilde \partial_s^\top} \tilde \partial_s^\top \right> - \left< \left[ \tilde \partial_t, \tilde \partial_s^\perp\right],\tilde \partial_s^\top \right>.
	\end{aligned}
	\end{equation}
	By (\ref{ineq-integrable-1}) and (\ref{ineq-integrable-2}),
	we conclude that for $s=0$,
	$$|\tilde \partial_s^\top|'(t)\le e^\epsilon r |\sndf| |\partial_s^\top| + e^{3\epsilon}|A|r \left[a(1-t)+bt+(a+b)o(r)\right].$$
	By integration,
	$$|\tilde \partial_s^\top|(t) \le e^{e^\epsilon r|\sndf|t}\cdot \int_0^te^{3\epsilon-e^\epsilon r |\sndf|\tau}\left[a(1-\tau)+b\tau+(a+b)o(r)\right] d\tau\cdot |A|r.$$ 
	Therefore 
	$$|\tilde \partial_s^\top|(0,1)\le (a+b) Ce^{3\epsilon+e^\epsilon|\sndf|r}|A|r.$$
\end{proof}

\begin{remark}
	(1.2) can be viewed as an extension of Lemma 3.3 in \cite{Ta2000} that was adopted to a Riemannian submersion, where the $T$-tensor and $A$-tensor defined by O'Neill \cite{ONeill1966} were used. Note that it is a standard fact (e.g., see \cite{Xu2010phdthesis}) that the $T$-tensor and $A$-tensor for a Riemannian submersion are equivalent to the second fundamental form  $\sndf_f$ of fibers and the integrability tensor $A_f$ of horizontal distribution in this paper.
\end{remark}
\begin{remark}\label{rem-dvarphi-ricci}
	Because the lack of control on a Jacobi field from the upper bound of Ricci curvature, the Rauch comparison theorem under lower bounded Ricci curvature provided in \cite{DaiWei1995} is not enough to derive (\ref{ineq-integrable-1}) under $|\Ric(N,h)|\le (n-1)$. Therefore the estimate in (1.2) generally fails to hold if the curvature condition is weakened to a Ricci curvature bound.
\end{remark}

\section{Proof of Theorem \ref{thm-tech}}

From now on we are to prove Theorem \ref{thm-tech}. Let $f_i:(M,g)\to (N,h)$ be $\epsilon$-Riemannian submersions satisfying the assumptions of Theorem \ref{thm-tech}. A bundle map $\Phi$ can be defined naturally as follows.
For small $0\le \epsilon\le r< \injrad(N,h)$ and for any $x\in M$, let $p=f_1(x)$ and $q=f_2(x)$, then the image of the bundle map $\Phi(x)$ is defined by
\begin{equation}\label{def-bundle-map}
\Phi(x)=\tilde \gamma_x(1),
\end{equation} where $\tilde \gamma_x:[0,1]\to M$ is the $f_2$-horizontal lifting curve of the unique geodesic $\gamma_x:[0,1]\to N$ from $\gamma_x(0)=q=f_2(x)$ to $\gamma(1)=p=f_1(x)$. Then $\tilde \gamma_x(t)\in f_2^{-1}(\gamma_x(t))$ depends smoothly on $x$. Hence, the map $\Phi$ together with the homotopy $H:[0,1]\times M\to M$, $H(t,x)=\tilde\gamma_x(t)$, are smooth maps such that $\Phi=H_1$ is homotopic to $H_0=\operatorname{Id}_M$, and $\Phi$ is a bundle map, i.e., $f_2\circ\Phi=f_1$.
If $\dif\Phi$ is isomorphic at every point $x\in M$, then $\Phi$ is a covering map from $M$ to itself homotopic to the identity, and hence a diffeomorphic bundle map. 

We first give a sufficient and necessary condition for that $\dif \Phi$ is an isomorphism.
Let $\varphi:f_2^{-1}(B_r(p))\to B_r(p)\times f_2^{-1}(p)$ be a local trivialization of $f_2$ centered at $f_2^{-1}(p)$, which is defined by
\begin{equation}\label{def-local-trivialization}
\varphi=(f_2(x),\varphi_2(x))=(f_2(x),\tilde\gamma_x(1)),\quad \text{for any $x\in f_2^{-1}(B_r(p))$,}
\end{equation} where $\tilde \gamma_x:[0,1]\to M$ is the $f_2$-horizontal lifting as in the construction of $\Phi$ above. Then by definitions (\ref{def-bundle-map}) and (\ref{def-local-trivialization}), 
\begin{equation}\label{bundle-map-coincides}
\Phi(x)=\varphi_2(x), \quad \text{for any $x\in f_1^{-1}(p)$.}
\end{equation}

\begin{lemma}\label{lem-transversal}
	$\dif \Phi$ is isomorphic at $T_xM$ if and only if the $f_1$-vertical distribution at $x$, $\mathcal V_{f_1}(x)$, is transversal to the radially horizontal slice $S_{\varphi_2(x)}=\varphi^{-1}(B_r(p)\times \varphi_2(x))$, i.e.,
	$\dif\varphi_2(w)\neq 0$ for any $w\in \mathcal V_{f_1}(x) \neq 0$.
\end{lemma}
\begin{proof}
	First, if $\dif\Phi(w)=0$ for some $0\neq w\in T_xM$, then $w\in \mathcal V_{f_1}(x)$. Indeed, if $w\not\in \mathcal V_{f_1}(x)$, then $\dif f_2\circ \dif\Phi(w)=\dif f_1(w)\neq 0$, which implies that $\dif \Phi(w)\neq 0$. Secondly, by (\ref{bundle-map-coincides}), $\dif \Phi|_{\mathcal V_{f_1}(x)}=\dif \varphi_2|_{\mathcal V_{f_1}(x)}$.
	Therefore for any $w\in \mathcal V_{f_1}(x)$, $\dif \Phi(w)=0$ if and only if $\dif \varphi_2(w)=0$, i.e., $w$ is tangent to the slice $S_{\varphi_2(x)}$.
\end{proof}

The key estimate in proving that $\dif \Phi$ is an isomorphism is the following closeness of $f_{i}$-vertical distributions. 
\begin{proposition}\label{prop-vertical-close}
	Under the assumption of Theorem \ref{thm-tech}, the vertical subspaces $\mathcal V_{f_1}(x)$ and $\mathcal V_{f_2}(x)$ are close in $T_xM$ in the sense that the dihedral angle between the subspaces $\mathcal V_{f_1}(x)$ and $\mathcal V_{f_2}(x)$ $\le \varkappa(\epsilon\,|\,c_0,r_0,m)$.
\end{proposition}
The \emph{dihedral angle} between $\mathcal V_{f_1}(x)$ and $\mathcal V_{f_2}(x)$ is defined to be the Hausdorff distance $d_H(\mathcal V_{f_1}(x)\cap T_x^1M,\mathcal V_{f_2}(x)\cap T_x^1M)$ in the unit sphere  $T_x^1M\subset T_xM$. The proof of Proposition \ref{prop-vertical-close} is left to the next section.
\begin{remark}\label{rem-not-C1}
	We do not know whether $f_1$ and $f_2$ in Theorem \ref{thm-tech} are $\varkappa(\epsilon)$-$C^1$-close in the sense of (\ref{ineq-C1-close}), mainly due to the lack of control on the twist of $S_{\varphi_2(x)}$. It should be pointed out that the $C^1$-closeness (\ref{ineq-C1-close}) was crucial to the earlier proofs in  \cite{Peters1984} (cf. \cite{Wa1992,Wa1993}) of (parametrized) diffeomorphic finiteness under bounded sectional curvature.
\end{remark}
By assuming Proposition \ref{prop-vertical-close}, we continue the proof of Theorem \ref{thm-tech}. 
\begin{proposition}\label{prop-almost-isometry}
	The estimate (\ref{ineq-almost-isometry}) holds for the bundle map $\Phi:M\to M$ defined in (\ref{def-bundle-map}), provided that $|\sec(N,h)|\le 1$.
\end{proposition}
\begin{proof}
	Firstly, let us prove that (\ref{ineq-almost-isometry}) holds for any $f_1$-vertical unit vector $v\in \mathcal V_{f_1}(x)$.  By the definition of $\Phi$, $\dif \Phi|_{\mathcal V_{f_1}(x)}=\dif \varphi_2|_{\mathcal V_{f_1}(x)}$. Hence it suffices to estiamte $\dif \varphi_2(v)$.  Let $v=v^{\bot}+v^{\top}$ be the orthogonal decomposition of $v$ such that $v^\bot\in \mathcal H_{f_2}(x)$ and $v^\top\in \mathcal V_{f_2}(x)$.
	By Proposition \ref{prop-vertical-close}, 
	\begin{equation}\label{ineq-prop-C1-close1}
	|v^\bot|\le \varkappa(\epsilon),\quad |v^\top|\ge \sqrt{1-\varkappa^2(\epsilon)}.
	\end{equation}
	At the same time, it follows from (1.1) and (1.2) in Lemma \ref{lem-variation}  that
	\begin{equation}\label{ineq-prop-C1-close2}
	|\dif \varphi_2(v^\top)|\ge e^{-c_0 e^\epsilon}|v^\top|,\quad |\dif \varphi_2(v^\perp)|\le
	Ce^{3\epsilon+e^\epsilon c_0}c_0\cdot |v^\perp|,
	\end{equation}
	where $C$ is a universal constant.
	Combing (\ref{ineq-prop-C1-close1}-\ref{ineq-prop-C1-close2}), we derive
	\begin{equation}\label{estimate-upper-1}
	\begin{aligned}
	|\dif \Phi(v)|&=|\dif \varphi_2(v)|\le |\dif \varphi_2(v^\top)|+|\dif \varphi_2 (v^\bot)|\\
	&\le e^{c_0e^\epsilon}+C e^{3\epsilon+e^\epsilon c_0}c_0\cdot \varkappa(\epsilon),
	\end{aligned}
	\end{equation}
	and \begin{equation}\label{estimate-lower-1}
	\begin{aligned}
	|\dif \Phi(v)|&\ge |\dif \varphi_2(v^\top)|-|\dif \varphi_2 (v^\bot)|\\
	&\ge e^{-c_0 e^\epsilon}\sqrt{1-\varkappa(\epsilon)}-Ce^{3\epsilon+e^\epsilon c_0}c_0\cdot \varkappa(\epsilon)\\
	&>0
	\end{aligned}
	\end{equation}
	as $\epsilon$ sufficient small. By (\ref{estimate-upper-1}-\ref{estimate-lower-1}), $\dif\Phi$ is isomorphic, and (\ref{ineq-almost-isometry}) holds along $f_1$-horizontal distribution.
	
	Next, we estimate $|\dif \Phi (u)|$ for any $f_1$-horizontal unit vector $u\in \mathcal H_{f_1}(x)$.
	Because $f_1$ and $f_2$ are $\epsilon$-Riemannian submersions, and
	$$e^{-\epsilon}\le |\dif f_2\circ \dif \Phi(u)|=|\dif f_1(u)|\le e^{\epsilon},$$
	we derive 
	\begin{equation}\label{estimate-twosided}
	e^{-2\epsilon}\le |\dif \Phi(u)^\bot|\le e^{2\epsilon}.
	\end{equation}
	By (1.2) in Lemma \ref{lem-variation}, 
	\begin{equation}\label{estiamte-upper-2}
	\begin{aligned}
	|\dif \Phi(u)^\top|&\le C(|\dif f_1(u)|+|\dif f_2(u)|)e^{3\epsilon+e^\epsilon c_0}c_0\\
	&\le 2 C e^{4\epsilon+e^\epsilon c_0} c_0.
	\end{aligned}
	\end{equation}
	By (\ref{estimate-twosided}-\ref{estiamte-upper-2}), $\dif \Phi(u)$ is almost $f_2$-horizontal (depending on $c_0$), and (\ref{ineq-almost-isometry}) also holds along $f_1$-horizontal distribution. Combing with the fact that $\dif \Phi(\mathcal V_{f_1})\subset \mathcal V_{f_2}$, we conclude that (\ref{ineq-almost-isometry}) holds for any vector.
\end{proof}

The only difference between Theorem \ref{thm-tech} and Proposition \ref{prop-almost-isometry} lies in the curvature condition on the base space. As the final step in proving Theorem \ref{thm-tech}, we apply the smoothing technique in \cite{DWY1996} via the Ricci flow \cite{Hamilton1982} to reduce the proof of Theorem \ref{thm-tech} to Proposition \ref{prop-almost-isometry}.

\begin{proof}[Proof of Theorem \ref{thm-tech}]	$~$
	
	Let $f_i:(M,g)\to (N,h)$ ($i=1,2$) be two $\epsilon$-Riemannian submersions in Theorem \ref{thm-tech}.
	According to Theorem 1.1 in \cite{DWY1996},
	there is $T(n,r_0)>0$ such that the solution $h(t)$ of Ricci flow
	$$\frac{\partial}{\partial t}h=-2\Ric(h),\quad h(0)=h$$
	exists in $[0,T(n,r_0)]$ such that the $C^0$-norm
	$$|h(t)-h(0)|_{h(0)} \le 4t, \quad 
	|\operatorname{Rm}(h(t))|_{h(t)} \le C(n, r_0)t^{-1/2}.
	$$
	and $\injrad(N,t^{-1/2}h(t))\ge r_1(n,r_0)$. For any $0<\epsilon\le T(n,r_0)$, let $\bar h_\epsilon = \epsilon^{-1} h(\epsilon)$ and $\bar g=\epsilon^{-1}g$. Then $f_i:(M,\bar g)\to (N,\bar h_\epsilon)$ are $[\epsilon-\ln (1-4\epsilon)]$-Riemannian submersions satisfying 
	$$\begin{aligned}
	&|\bar h_\epsilon-\epsilon^{-1}h|_{\epsilon^{-1}h}\le 4\epsilon, \; &|\sec(N,\bar h_\epsilon)|\le \epsilon^{1/2}C(n,r_0), \\
	&d_{\bar h_\epsilon}(f_1,f_2)\le \epsilon^{1/2}(1+4\epsilon), \; 
	&\injrad(N,\bar h_\epsilon)\ge \epsilon^{-1/4}r_1(n,r_0).
	\end{aligned}$$
	Because the closeness of vertical distributions $\mathcal V_{f_1}$ and $\mathcal V_{f_2}$ is rescaling invariant, Proposition \ref{prop-vertical-close} still holds for $\bar g$.
	Let $\Phi:M\to M$ be defined as (\ref{def-bundle-map}) with respect to $\bar h$. Then by Proposition \ref{prop-almost-isometry}, $\Phi$ is a diffeomorphism and the desired estimate on $|\dif \Phi|$ holds for the rescaled metric $\bar g$. Because (\ref{ineq-almost-isometry}) is rescaling invariant, the proof of Theorem \ref{thm-tech} is complete.
\end{proof}

It is clear from the construction of $\Phi$ in Theorem \ref{thm-tech}  that, if $f_1$ and $f_2$ are equivariant under some isometric actions of a closed Lie group $G$ on $M$ and $N$ respectively, then $\Phi$ is also $G$-equivariant. That is, the equivariant version of Theorem \ref{thm-tech} holds. Moreover, after smoothing the base spaces via the same method (\cite{DWY1996}) in the proof of Theorem \ref{thm-tech},  Theorem \ref{thm-equivariant} follows directly from the standard facts on equivariant convergence (e.g., see \cite{Ro2010}) and the equivariant version of Theorem \ref{thm-tech}. Here we omit its detailed proof.

\section{Closeness of the Vertical Distributions}

The remaining of the paper is devoted to prove Proposition \ref{prop-vertical-close}. In preparation we lift $g$ and $V_{f_i}(x)$ ($i=1,2$) to the tangent space $T_xM$.
Let $\exp_{x}:T_{x}M\to M$ be exponential map of $(M,g)$. 
Let $g^*=\exp_{x}^*g$ be the pullback tensors on $T_{x}M$.
Because $\conjrad(M,g)\ge r_0$, it is well known (e.g. see \cite{Xu17}) that the ball $U=B_{\frac{r_0}{3}}(o)\subset T_{x}M$ satisfies  
$$\injrad(U,g^*)\ge \frac{2r_0}{3}, \quad   \left|\Ric(U,g^*)\right|\le 1.$$
The lifting $\epsilon$-Riemannian submersions
$$\tilde f_{i}=f_{i}\circ \exp_{x}:(U,g^{*})\to \tilde f_i(U)\subset (N,h)$$
are $\epsilon$-close in the sense that $d_{h}(\tilde f_{1},\tilde f_{2})\le \epsilon,$ and the dihedral angle between vertical subspaces of $\tilde f_i$ coincides with that of $f_i$.

\begin{proof}[Proof of Proposition \ref{prop-vertical-close}] $~$
	
	We argue by contradiction. Assume that there is $\theta>0$ such that for any $\epsilon_j>0$, there are $\epsilon_{j}$-Riemannian submersions $f_{i,j}:(M_j,g_{j})\to (N_j,h_j)$ $(i=1,2)$ satisfying the assumptions of Theorem \ref{thm-tech} such that $\dim M_j=m$, $\dim N_j=n$, and  the dihedral angle between the vertical subspaces of $\mathcal V_{f_{i,j}}(x_j)$ of $f_{i,j}$ at some $x_j\in M_j$ is no less that $\theta$.

	By the $C^{1,\alpha}$-convergence Theorem \ref{thm-anderson-compactness} (or by the arguments in the proof of Main Lemma 2.2 in \cite{Anderson1990}), the blow up sequence of the pointed tangent spaces $(U_j,o_j,\epsilon_j^{-2}g^{*}_{j})$ converges to the Euclidean space $\mathbb{R}^m$ in $C^{1,\alpha}$-topology,
	\begin{equation}\label{vertical-close-C1-convergence}
	(U_j, o_j, \epsilon_j^{-2}g^{*}_{j})\overset{C^{1,\alpha}}{\longrightarrow}(\mathbb{R}^m,o), \qquad j\to \infty.
	\end{equation}
	Since after blow up, the map 
	$$\tilde f_{i,j}:(U_j,o_j,\epsilon_j^{-2}g^{*}_{j})\to 
	\left(\tilde f_{i,j}(U_j),f_{i,j}(x_j),\epsilon_j^{-2}h_j\right)$$ are still $\epsilon_j$-Riemannian submersions, by Arzel\`a-Ascoli theorem and passing to a subsequence, we may assume that $\tilde f_{i,j}$ converges to a submetry $\tilde f_{i,\infty}:\mathbb{R}^n\times \mathbb{R}^{m-n}\to \mathbb{R}^n$ ($i=1,2$) (whose base points on $\mathbb R^n$ may be different). 
	It is well known that any submetry between Euclidean spaces is a canonical projections; e.g., see \cite{JiangLiXu17}. Since after blow up
	$$d_{\epsilon_{j}^2h_j}(\tilde f_{1,j},\tilde f_{2,j})\le 1,$$
	the two limit projections satisfy
	$$d(\tilde f_{1,\infty},\tilde f_{2,\infty})\le 1.$$
	It follows that after a translation by $\xi\in \mathbb R^n$ with $|\xi|\le 1$, 
	\begin{equation}\label{eq-limit-proj}
	\tilde f_{2,\infty}=\tilde f_{1,\infty}+\xi.
	\end{equation}
	In particular, the fibers of $\tilde f_{1,\infty}$ are parallel to that of $\tilde f_{2,\infty}$.
	
	Let $(x^1,\dots,x^k,\dots,x^m)$ the Cartesian coordinates on $\mathbb R^m$. By the $C^{1,\alpha}$-convergence (\ref{vertical-close-C1-convergence}), the pullback metrics from $U_j$ to a fixed large ball $B_R(o)\subset \mathbb R^m$, denoted by $g^{*}_{\epsilon_j}$, satisfy that 
	\begin{equation}\label{eq-vertical-close-gij}
	g^{*}_{\epsilon_j}(\partial{x^k},\partial{x^l})\to  \delta_{kl},\quad \Gamma_{kl}^{s}\to 0, \quad \text{as }j\to \infty,
	\end{equation}
	where $\Gamma_{kl}^s$ is the Christoffel symbols. 
	In the following we identify $U_j$ as a subspace of the limit space $\mathbb R^m$, and view the tangent space of $\tilde f_{i,j}$-fiber passing through $o_j$ as a subspace $P_{i,j}\subset \mathbb R^m$ defined by
	$$P_{i,j}=\{tv_{i,j}\,|\, t\in\mathbb R, v_{i,j}\in \mathcal V_{\tilde f_{i,j}}(o_j)\}\cap U_j.$$
	
	\begin{claim}\label{claim-tangent-plane}
		The tangent plane $P_{i,j}$ of the $\tilde f_{i,j}$-fiber
		converges to that of $\tilde f_{i,\infty}$ as $j\to\infty$.
	\end{claim}
	Clearly, by the claim and (\ref{eq-limit-proj}), the dihedral angle between $\mathcal V_{f_{i,j}}$ goes to $0$, a contradiction to the choice of $x$. Therefore, what remains is to verify Claim \ref{claim-tangent-plane}.

	Let $\alpha_{i,j}$ be any unit-speed geodesic in the submanifold $\tilde f_{i,j}^{-1}(\tilde f_{i,j}(o_j))$ starting at $o_j\in U_j$, and let $\beta_{i,j}(t)=t\alpha_{i,j}'(0)$ be the line with respect to the Cartesian coordinates $(x^1,\dots,x^k,\dots,x^m)$.
	By (4.2), the norm of the second fundamental form of $\tilde f_{i,j}$-fiber $|\sndf_{\tilde f_{i,j}}|\le c_0$, which implies that $\alpha_{i,j}$ is a $c_0$-almost geodesic, i.e.,
	\begin{equation}\label{ineq-vertical-close-almost-geodesic-1}
	|\nabla_{\alpha_{i,j}'}\alpha_{i,j}'|\le c_0.
	\end{equation}
	At the same time, by (\ref{eq-vertical-close-gij}), the geodesic curvature of $\beta_{i,j}$
	\begin{equation}\label{ineq-vertical-close-almost-geodesic-2}
	|\nabla_{\beta_{i,j}'}\beta_{i,j}'|\to 0, \qquad j\to \infty.
	\end{equation}
	By (\ref{eq-vertical-close-gij})-(\ref{ineq-vertical-close-almost-geodesic-2}), we are able to apply Proposition \ref{prop-almost-geodesic-deviation} in the next section, and thus by (\ref{eq-deviation-1}) the deviation of $\beta_{i,j}(t)$ from $\alpha_{i,j}$ satisfies
	$$d(\beta_{i,j}(t),\alpha_{i,j})\le \varkappa(c_0)t^{\frac{5}{4}}+\varkappa(\epsilon_j\,|\,m)t^{\frac{3}{2}}.$$Let $j\to \infty$, it follows that the distance from the limit line $\beta_{i,\infty}$ of $\beta_{i,j}$ to the limit subplane $\tilde f_{i,\infty}^{-1}(\tilde f_{i,\infty}(o))$ of $\tilde f_{i,j}^{-1}(f_{i,j}(o_j))$,
	\begin{equation}\label{ineq-vertical-close-almost-geodesic-3}
	d(\beta_{i,\infty}(t), \tilde f_{1,\infty}^{-1}(\tilde f_{i,\infty}(o)))\le \varkappa(c_0)t^{\frac{5}{4}}.
	\end{equation}
	Because the order in (\ref{ineq-vertical-close-almost-geodesic-3}) is higher than linear, $\beta_{i,\infty}$ must lie in the subplane $\tilde f_{i,\infty}^{-1}(\tilde f_{i,\infty}(o))$. Hence the pointed Hausdorff limit of $P_{i,j}$ coincides with $\tilde f_{i,\infty}^{-1}(\tilde f_{i,\infty}(o))$. Now the proof of Claim \ref{claim-tangent-plane} is complete.
\end{proof}

\section{Deviation of two Curves by their Geodesic Curvature}

In the proof of Proposition \ref{prop-vertical-close} a uniform and explicit estimate on the deviation of two curves that depends on their geodesic curvature plays a crucial role. Because we cannot find a reference in literature and it is of some independent interest, we present a proof that is due to Zuohai Jiang and the second author.

\begin{proposition}\label{prop-almost-geodesic-deviation}
	Let $\alpha,\beta$ be two unit-speed curves in a complete Riemannian $m$-manifold $(M,g)$ such that $p=\alpha(0)=\beta(0)$, $\alpha'(0)=\beta'(0)$ and the geodesic curvature $$|\nabla_{\alpha'}\alpha'|\le \delta_1,\quad |\nabla_{\beta'}\beta'|\le \delta_2.$$  If there is a local coordinates system $\{U,(x^{1}\cdots, x^{m})\}$ around $p$ such that the metric and the Christoffel symbols in $U$ satisfy
	$$C^{-1}I\le (g_{ij})\le CI,\quad |\Gamma_{ij}^{k}|\le\mu\;(i,j,k=1,\cdots, m),$$ then the distance from $\alpha(s)$ to the curve $\beta$, $$r(s)=\operatorname{dist}_{\beta}\circ\alpha(s)=\min_t\{d(\alpha(s),\beta(t))\}$$ satisfies
	\begin{eqnarray}\label{eq-deviation-1}
	r(s)\le \varkappa(\delta_{2})s^{\frac{5}{4}}+\varkappa(\delta_1,\mu\,|\,C,m)s^{\frac{3}{2}}, 
	\end{eqnarray}
	provided that, for any $0\le s_1\le s$, $r(s_1)$ is realized by a minimal geodesic connecting $\alpha(s_1)$ and $\beta$ which lies in $U$.
\end{proposition}
\begin{proof}
	Let us choose an orthonormal basis $\{e_{j}\}_{j=1}^m$ in $T_{p}M$ such that $e_{1}=\alpha'(0)$. Let $e_{j,\alpha}(s)$ ($j=1,\cdots, m$) be the parallel vector fields along $\alpha$ such that $e_{j,\alpha}(0)=e_j$. For $x\in M$, let  $h(x)=\frac{1}{2}\operatorname{dist}_\beta^{2}(x)$. If $h$ is differentiable at $\alpha(s)$, then
	\begin{align}
	(h\circ \alpha)'(s)&= \langle\nabla h, \alpha'\rangle=\sum\limits_{j=1}^{m}\langle\alpha',e_{j,\alpha}(s)\rangle \langle\nabla h, e_{j,\alpha}(s)\rangle \nonumber\\ 
	&=\langle\alpha',e_{1,\alpha}(s)\rangle \langle\nabla h, e_{1,\alpha}(s)\rangle+\sum\limits_{j=2}^{m} \langle\alpha',e_{j,\alpha}(s)\rangle \langle\nabla h, e_{j,\alpha}(s)\rangle \label{eq-deviation-2}
	\end{align}
	Note that for $j=2,\cdots, m$,
	$$\begin{aligned}
	\langle\alpha', e_{j,\alpha}(s)\rangle&=\langle\alpha'(0), e_{j,\alpha}(0)\rangle+\int_{0}^{s}\langle\alpha',  e_{j,\alpha}\rangle' (t)dt\\
	&= 
	\langle e_1, e_j\rangle+\int_{0}^{s} 
	\langle\nabla_{\alpha'}\alpha', e_{j,\alpha}\rangle(t) dt\\
	&\leq \delta_{1} s.
	\end{aligned}$$
	Therefore the last term in (\ref{eq-deviation-2}) satisfies
	\begin{equation}
	\label{eq-deviation-3}
	\begin{aligned}
	\sum\limits_{j=2}^{m}\langle\alpha',e_{j,\alpha}(s)\rangle \langle\nabla h, e_{j,\alpha}(s)\rangle
	&\le \delta_{1}s\left(\sum\limits_{j=2}^{m}\left|\langle r\nabla r, e_{j,\alpha}(s)\rangle\right|\right) \\
	&\le (m-1)\delta_{1}s^{2}.
	\end{aligned}
	\end{equation}
	
	We now estimate $\langle\nabla h, e_{1,\alpha}(s)\rangle$ in the first term of (\ref{eq-deviation-2}). Because $h=\frac{1}{2}\operatorname{dist}_{\beta}^2$ is differentiable at $\alpha(s)$, the unit-speed minimal geodesic $\gamma(t)$ from $\beta$ to $\alpha(s)$ whose length realizes $r(s)$ is unique. Assume that $\gamma(0)=\beta(s_0)$ and $\gamma(r(s))=\alpha(s)$. Let $\tilde e_{1,\beta}(s)$ be the parallel transport of $e_1=\beta'(0)$ along $\beta(s)$, and $\tilde e_{1,\gamma}(t)$ be the parallel transport of $\tilde e_{1,\beta}(s_0)$ along $\gamma$. 
	Observe that
	\begin{align}
	\langle\nabla h, e_{1,\alpha}(s)\rangle &=
	\left<r\gamma'(r), e_{1,\alpha}(s)- \tilde e_{1,\gamma}(r)\right>+ 	\left<r\gamma'(r), \tilde e_{1,\gamma}(r)\right>\nonumber\\
	\label{eq-deviation-4}
	&\le s \left| e_{1,\alpha}(s)-\tilde e_{1,\gamma}(r) \right| + s \left<\gamma'(r), \tilde e_{1,\gamma}(r)\right>
	\end{align}
	Since $\alpha(s)$, $\beta(s)$ and $\gamma(t)$ lie in $U$, by Lemma \ref{lem-parallel-transport} below, the holonomy of parallel transport along the closed curve $\beta*\gamma*\alpha^{-1}$ is small, i.e.,
	\begin{equation}\label{eq-deviation-5}
	\left| e_{1,\alpha}(s)-\tilde e_{1,\gamma}(r) \right|\le \sqrt{m^5C^3}\mu\cdot s.
	\end{equation}
	For the term $\left<\gamma'(r), \tilde e_{1,\gamma}(r)\right>$ in (\ref{eq-deviation-4}), 
	notice that
	$$\left|\left<\tilde e_{1,\beta}(s), \beta'(s) \right>-1\right|
	\le \int_0^s \left|\left<\tilde e_{1,\beta}, \nabla_{\beta'}\beta'\right>(t)\right| dt \le \delta_2s,$$
	and $\left<\beta'(s_0), \gamma'(0)\right>=0$,
	we derive
	\begin{equation}\label{eq-deviation-6}
	\begin{aligned}
	\left< \gamma'(t), \tilde e_{1,\gamma}(t) \right> &=
	\left< \gamma'(0), \tilde e_{1,\beta}(s_0)  \right>\\
	&=\left<\gamma'(0), \tilde e_{1,\beta}(s_0)-\beta'(s_0)\right>\\&
	\le \sqrt{2\delta_2s_0}.
	\end{aligned}
	\end{equation}
	By (\ref{eq-deviation-5}) and (\ref{eq-deviation-6}),
	\begin{equation}\label{eq-deviation-7}
	\langle\nabla h, e_{1,\alpha}(s)\rangle \le \sqrt{m^5C^3}\mu s^2 + 2\sqrt{\delta_2}s^{\frac{3}{2}}
	\end{equation}
	Combing (\ref{eq-deviation-2}), (\ref{eq-deviation-3}) and (\ref{eq-deviation-7}), we have	
	\begin{equation}\label{eq-deviation-end}
	(h\circ\alpha)'(s)\le 2\sqrt{\delta_2}s^{\frac{3}{2}} + \left[(m-1)\delta_{1}+\sqrt{m^5C^3}\mu\right]s^{2}
	\end{equation}
	
	If $h$ is not differentiable at $\alpha(s)$, then let us choose a unit-speed minimal geodesic $\gamma$ from $\beta(s_0)$ to $\alpha(s)$ whose length realizes the distance $r(s)$, and consider the barrier function from above $\tilde h=\frac{1}{2}(\operatorname{dist}_{\gamma(\epsilon)}+\epsilon)^2$. Then $\tilde h$ is smooth at $\alpha(s)$ and the right derivative $(h\circ\alpha)^+(s)\le (\tilde h\circ \alpha)'(s)$. By constructing the same parallel vector fields $\tilde e_{1,\beta}$ and $\tilde e_{1,\gamma}$ along $\beta$ and $\gamma$ respectively, it follows from the same argument as the above that (\ref{eq-deviation-end}) holds.
	
	Because $h\circ\alpha$ is Lipschitz, by integrating (\ref{eq-deviation-end})
	$$h\circ\alpha(s)\le \frac{1}{3}\left[\sqrt{m^5C^3}\mu+(m-1)\delta_1\right]s^3 +\frac{4}{5}\sqrt{\delta_2}s^{\frac{5}{2}}$$
\end{proof}

In the end of the paper we prove the elementary lemma about the holonomy of parallel transport used in the proof of Proposition \ref{prop-almost-geodesic-deviation}. 
\begin{lemma}[cf. \cite{Post}]\label{lem-parallel-transport}
	Let $\alpha:[0,l]\to M$ be a unit-speed piecewisely smooth closed curve lying in a coordinates neighborhood $\{U,(x^1,\cdots,x^m)\}$. If the Riemannian metric $g_{ij}=g(\frac{\partial}{\partial x^{i}},\frac{\partial}{\partial x^j})$ and the Christoffel symbols $$C^{-1}I\le (g_{ij})\le CI, \quad |\Gamma_{ij}^{k}|\le \mu \quad 
	(i,j,k=1,\cdots, m),$$ then any unit parallel vector field $v(s)$ along $\alpha$ satisfies
	\begin{equation}\label{eq-parallel-1}
	|v(0)-v(l)|^2\le m^5 C^3\mu^2 \cdot l^2.
	\end{equation}
\end{lemma}
\begin{proof}
	Let us denote $\alpha(s)=(x^{1}(s),\cdots, x^{m}(s))$, and $v(s)=\sum_i v^{i}(s)\frac{\partial}{\partial x^{i}}$. Then $\nabla_{\alpha'}v(s)=0$ implies that
	$$
	(v^{k})'(s)=-\sum_{ij}v^{i}(s)(x^{j})'(s)\Gamma_{ij}^{k}(s). \quad (k=1,\cdots, m)
	$$
	By integrating the above equation, we derive
	\begin{align}\label{eq-parallel-2}
	\left|v^{k}(l)-v^{k}(0)\right|&=
	\left|\sum_{ij}\int^{l}_{0}v^{i}(s)(x^{j})'(s)\Gamma_{ij}^k(s)ds\right|\\ \nonumber
	&\leq \sum_{ij} \int^{l}_{0}\left|v^{i}(s)(x^{j})'(s)\Gamma_{ij}^k(s)\right|ds
	\end{align}
	Since
	$$
	(v^i(s))^2\le C\cdot \left|v(s)\right|^{2}=C, \quad	((x^i)'(s))^2\le C\cdot \left|\alpha'(s)\right|^{2}=C
	$$
	and $|\Gamma_{ij}^{k}|\le\mu$ $(i,j,k=1,\cdots, m)$, it follows from (\ref{eq-parallel-2}) that 
	\begin{equation}\label{eq-parallel-end}
	\left|v^{k}(l)-v^{k}(0)\right|\leq m^2 C\mu\cdot  l\quad  (k=1,\cdots, m).
	\end{equation}
	Now (\ref{eq-parallel-1}) follows directly from (\ref{eq-parallel-end}) and $(g_{ij})\le CI$.
\end{proof}

\bibliographystyle{plain}
\bibliography{document}{}

\begin{thebibliography}{10}

\bibitem{Anderson1990}
Michael~T. Anderson.
\newblock {Convergence and rigidity of manifolds under Ricci curvature bounds}.
\newblock {\em Inventiones mathematicae}, 102(1):429--445, 1990.

\bibitem{Berestovskii2000}
V.~N. Berestovskii and Luis Guijarro.
\newblock A metric characterization of riemannian submersions.
\newblock {\em Annals of Global Analysis and Geometry}, 18(6):577--588, Dec
  2000.

\bibitem{Cheegerphdthesis}
J.~Cheeger.
\newblock {\em Comparison and finiteness theorems for Riemannian manifolds}.
\newblock PhD thesis, Princeton University, 1967.

\bibitem{CFG1992}
J.~Cheeger, K.~Fukaya, and M.~Gromov.
\newblock Nilpotent structures and invariant metrics on collapsed manifolds.
\newblock {\em Journal of the American Mathematical Society}, 5(2):327--372,
  1992.

\bibitem{Cheeger1970Finiteness}
Jeff Cheeger.
\newblock Finiteness theorems for {R}iemannian manifolds.
\newblock {\em American Journal of Mathematics}, 92(1):61--74, 1970.

\bibitem{DaiWei1995}
Xianzhe Dai and Guofang Wei.
\newblock {A comparison-estimate of Topogonov type for Ricci curvature}.
\newblock {\em Mathematische Annalen}, 303(2):297--306, 1995.

\bibitem{DWY1996}
Xianzhe Dai, Guofang Wei, and Rugang Ye.
\newblock Smoothing {R}iemannian metrics with {R}icci curvature bounds.
\newblock {\em manuscripta mathematica}, 90(1):49--61, 1996.

\bibitem{Fukaya1987Collapsing}
Kenji Fukaya.
\newblock Collapsing {R}iemannian manifolds to ones with lower dimensions.
\newblock {\em Journal of Differential Geometry}, 25:139--156, 1987.

\bibitem{Fukaya1988}
Kenji Fukaya.
\newblock A boundary of the set of the {R}iemannian manifolds with bounded
  curvatures and diameters.
\newblock {\em Journal of Differential Geometry}, 28:1--21, 1988.

\bibitem{GW1988}
R.~E. Green and H.~Wu.
\newblock Lipschitz convergence of {R}iemannian manifolds.
\newblock {\em Pacific J. Math.}, 131:119--141, 1988.

\bibitem{GLP1981}
Mikhael Gromov.
\newblock {\em Structures m\`etriques pour les vari\`et\`es riemanniennes.
  (French) [Metric structures for Riemann manifolds]}.
\newblock Textes Math¨¦matiques [Mathematical Texts], 1. CEDIC/Fernand Nathan,
  Paris, 1981.

\bibitem{Hamilton1982}
Richard~S. Hamilton.
\newblock Three-manifolds with positive {R}icci curvature.
\newblock {\em J. Differential Geom.}, 17(2):255--306, 1982.

\bibitem{JiangLiXu17}
Zuohai Jiang, Xiang Li, and Shicheng Xu.
\newblock Stability of nilpotent structures of collapsed manifolds on the same
  scale.
\newblock preprint, 2017.

\bibitem{Kapovitch2007Perelman}
V.~Kapovitch.
\newblock Perelman's stability theorem.
\newblock In {\em {Metric and Comparison Geometry}}, volume~XI of {\em Surveys
  in Differential Geometry}, pages 103--136. International Press of Boston,
  2007.

\bibitem{Kasue1989}
Atsushi Kasue.
\newblock A convergence theorem for {R}iemannian manifolds and some
  applications.
\newblock {\em Nagoya Math. J.}, 114:21--51, 1989.

\bibitem{ONeill1966}
Barrett O'Neill.
\newblock The fundamental equations of a submersion.
\newblock {\em Michigan Math. J.}, 13(4):459--469, 1966.

\bibitem{Per1991}
G.~Perelman.
\newblock Alexandrov spaces with curvatures bounded from below {II}.
\newblock {\em Preprint}, 1991.

\bibitem{Peters1984}
Stefan Peters.
\newblock {Cheeger's finiteness theorem for diffeomorphism classes of
  Riemannian manifolds}.
\newblock {\em Journal f\"ur die reine und angewandte Mathematik}, 349:77--82,
  1984.

\bibitem{Peters1987}
Stefan Peters.
\newblock Convergence of {R}iemannian manifolds.
\newblock {\em Compositio Mathematica}, 62(1):3--16, 1987.

\bibitem{Pe16}
Peter Petersen.
\newblock {\em Riemannian {G}eometry}, volume 171 of {\em Graduate Texts in
  Mathematics}.
\newblock Springer International Publishing, 3 edition, 2016.

\bibitem{Post}
M.~M. Postnikov.
\newblock {\em Geometry VI Riemannian Geometry}.
\newblock Springer-Verlag, Berlin Heidelberg, 2001.

\bibitem{Ro2010}
X.~Rong.
\newblock Convergence and collapsing theorems in riemannian geometry.
\newblock In {\em Handbook of Geometric Analysis}, volume II ALM 13, pages
  193--298. Higher Education Press and International Press, Beijing-Boston,
  2010.

\bibitem{Rong2011Stability}
X.~Rong and S.~Xu.
\newblock Stability of almost submetries.
\newblock {\em Front. Math. China}, 6(1):137--154, 2011.

\bibitem{Rong2012Stability}
X.~Rong and S.~Xu.
\newblock Stability of $e^\epsilon$-lipschitz and co-lipschitz maps in
  gromov-hausdorff topology.
\newblock {\em Advances in Mathematics}, 231(2):774--797, 2012.

\bibitem{Ta2000}
K.~Tapp.
\newblock Bounded {R}iemannian submersions.
\newblock {\em Indiana Univ. Math. J.}, 49(2):637--654, 2000.

\bibitem{Ta2002}
K.~Tapp.
\newblock Finiteness theorems for submersions and souls.
\newblock {\em Proc. Amer. Math. Soc.}, 130(6):1809--1817, 2002.

\bibitem{Wa1992}
P.~Walczak.
\newblock A finiteness theorem for {R}iemannian submersions.
\newblock {\em Ann. Polon. Math.}, 57:283--290, 1992.

\bibitem{Wa1993}
P.~Walczak.
\newblock Erratum to the paper {A finiteness theorem for Riemannian
  submersions}.
\newblock {\em Ann. Polon. Math.}, 58:319, 1993.

\bibitem{Wu1996}
J.~Y. Wu.
\newblock A parametrized geometric finiteness theorem.
\newblock {\em Indiana Univ. Math. J.}, 45(2):511--528, 1996.

\bibitem{Xu2010phdthesis}
Shicheng Xu.
\newblock {\em Stability theorems on almost submetries}.
\newblock PhD thesis, Capital Normal University, 2010.

\bibitem{Xu17}
Shicheng Xu.
\newblock Local estimate on convexity radius and decay of injectivity radius in
  a riemannian manifold.
\newblock {\em Communications in Contemporary Mathematics}, July 2017.

\end{thebibliography}

\end{document}